\documentclass[12pt]{amsart}
\usepackage{eurosym}
\usepackage{amsfonts}
\usepackage{a4wide}
\usepackage{amsmath}
\usepackage{amssymb}
\usepackage{amscd}
\usepackage{color}
\usepackage[dvipsnames]{xcolor}
\usepackage[all]{xy}
\usepackage{lineno,hyperref}
\usepackage{lineno}

\newtheorem{theorem}{Theorem}[section]

\newtheorem{lemma}[theorem]{Lemma}
\newtheorem{proposition}[theorem]{Proposition}
\theoremstyle{definition}
\newtheorem{definition}[theorem]{Definition}

\newtheorem{example}[theorem]{Example}

\newtheorem{remark}[theorem]{Remark}

\numberwithin{equation}{section}

\begin{document}
\author{A. Alahmadi}
\address{Department of Mathematics,
	King abdulaziz University,
	Jeddah, SA\\
	email:adelnife2@yahoo.com}
\author{S.K. Jain}
\address{Department of Mathematics, University of Ohio, Athens, USA \\
email: jain@ohio.edu}
\author{A. Leroy }
\address{Artois University,Laboratoire de Math\'ematiques de Lens, UR 2462 , F-62300, Lens,
France\\
email: andre.leroy@univ-artois.fr}
\title[]{Remarks on separativity of regular rings}

\begin{abstract}
Separative von Neumann regular rings exist in abundance. For example, all
regular self-injective rings, unit regular rings, regular rings with a
polynomial identity are separative. It remains open whether there exists a
non-separative regular ring. In this note, we study a variety of conditions
under which a von Neumann regular ring is separative. We show that a von
Neumann regular ring $R$ is separative under anyone of the following cases: {\it (1)} $R$ is CS; {\it (2)} $R$ is pseudo injective (auto-injective); {\it (3)} $R$ satisfies the closure extension property: the essential closures in $R$ of two isomorphic right ideals are themselves isomorphic.
 We also give another characterization of a regular perspective ring (Proposition \ref{unit regular})
\end{abstract}

\maketitle

\section{Introduction and preliminaries}

A ring $R$ is von Neumann regular if for any $a\in R$ there exists $b\in R$
such that $a=aba$. A ring $R$ is separative if for any finitely generated
projective right modules $M$ and $N$, the isomorphisms $M\oplus N\cong
M\oplus M\cong N\oplus N$ imply $M\cong N$. Many regular rings are
separative. For example, right (or left) self-injective regular rings, unit
regular rings, regular rings with polynomial identity are separative. The
separativity condition for regular rings is connected with the following two
important properties for a nonunit element $a\in R$:

\begin{equation}
l(a)R=Rr(a)=R(1-a)R
\end{equation}
where $r(a)=\{x \in R \mid ax=0\}$ and $l(a)=\{x\in R \mid xa=0\}$.

\begin{equation}
\mathrm{\ The \; element \; a\; is\; a \;product\; of\; idempotents.}
\end{equation}%
.

In \cite{HO} Hannah and O'Meara proved the following:

\begin{theorem}
Let $R$ be a regular right self-injective ring and let $a\in R$. Then $a$ is
a product of idempotents if and only if $Rr(a)=l(a)R=R(1 - a)R$.
\end{theorem}

The following remarkable equivalence is due to O'Meara (\cite{O}).

\begin{theorem}
Let $R$ be a regular ring. Then $R$ is separative if and only for every non
unit $a\in R$, conditions (1.1) and (1.2) are equivalent.
\end{theorem}

The right maximal quotient ring of a right nonsingular ring is a von Neumann
regular right self injective ring. In particular, the above two theorems
show that the right maximal quotient ring $Q$ of a regular ring is always
separative. So any regular ring $R$ can be embedded in a regular separative
ring viz. its maximal right quotient ring $Q(R)$. We are thus interested on
the question when the separativity property goes down from $Q(R)$ to $R$? In
the next section we will give some sufficient conditions for this to happen.

\bigbreak

\section{Main result}

\smallbreak
The following lemma is a key result that will be invoked at several places.

\begin{lemma}
\label{going to idempotent} Let $R$ be a regular ring and $Q=Q(R)$ be its
right maximal quotient ring. If for any finite set of idempotents, $%
e_1,\dots,e_n, f_1,\dots, f_l$ in $R$, $\prod_{i=1}^n e_iQ\cong
\prod_{j=1}^l f_jQ$ implies $\prod_{i=1}^n e_iR\cong \prod_{j=1}^l f_jR$,
then the ring $R$ is separative.
\end{lemma}

\begin{proof}
Suppose that, for any idempotents $e_1,e_2,\dots e_n$ and $f_1,f_2,\dots
,f_l $ of $R$ we have $\prod_{i=1}^n e_iQ \cong \prod_{j=1}^l f_jQ$ implies
that $\prod_{i=1}^n e_iR\cong \prod_{j=1}^l f_jR$. Since the ring $R$ is
regular, finitely generated projective right $R$-modules are isomorphic to
direct sums of principal ideals of the form $eR$. Suppose $M,N$ are two
right finitely generated projective $R$-modules such that 
\begin{equation*}
M\oplus N \cong M\oplus M\cong N\oplus N.
\end{equation*}

Tensoring with $Q=Q_{max}(R)$ over $R$, we get 
\begin{equation*}
M\otimes Q\oplus N\otimes Q \cong M\otimes Q\oplus M\otimes Q\cong N\otimes
Q\oplus N\otimes Q.
\end{equation*}
Since $M\otimes Q$ and $N\otimes Q$ are also finitely generated, the fact
that $Q$ is separative implies that $M\otimes Q \cong N\otimes Q $. Since
the ring $R$ is regular, we can write finitely generated projective right $R$%
-modules as $M\cong \prod_{i=1}^n e_iR$ and $N\cong\prod_{j=1}^l f_jR$. We
then have $M\otimes Q\cong \prod e_iR\otimes Q \cong\prod e_iQ$ and
similarly $N\otimes Q \cong \prod f_j Q$. Therefore $\prod_{i=1}^n e_i Q
\cong \prod_{j=1}^l f_jQ$. So by hypothesis $\prod_{i=1}^n e_iR \cong
\prod_{j=1}^l f_j R$ hence $M\cong N$ and the ring is separative.
\end{proof}

\begin{remark}
(see Lam Exercise 21.13 \cite{L2}) This may be of the interest for the
reader to know that if $R$ is a subring of a ring $Q$ (say for example right
maximal quotients of $R$) and if $eR\cong fR$ for any two idempotents $e,f$
in $R$ then $eQ\cong fQ$. However, the converse need not be true.
\end{remark}

For presenting the first case as when a regular ring is separative we need
the notion of pseudo injectve module now known as auto invariant modules. A
module $M$ is called pseudo injective if every monomorphism from a submodule 
$N$ of $M$ can be extended to an endomorphism of $M$ (\cite{JS}). It has
been shown in \cite{ESS} that $M$ is pseudo injcetive if and only if $M$ is
invariant under automorphism of $E(M)$ where $E(M)$ is the injective hull of 
$M$ (See the recent book \cite{STG} ). In view of this, pseudo injective
modules are called auto invariant modules. Note that auto invariant modules
need not be CS, as is the case for quasi continuous ($\pi$-injective), quasi
injective and injective modules. We illustrate this by first proving the
following lemma. This result is in proposition 3.2, Chapter 12 in \cite{S} .
We offer here a direct simple proof.

\begin{lemma}
The maximal quotient ring of a Boolean ring is again Boolean.
\end{lemma}

\begin{proof}
Let $B$ be a Boolean ring and $Q=Q(B)$ its right maximal quotient ring. As
is well-known and easy to check, $R$ and $Q$ are both commutative. For any $%
q\in Q$ there exists a right essential ideal $E$ of $R$ such that $qE\subset
B$. So for any $x\in E$, $(qx)^2=qx$. This gives $qx=q^2x^2=q^2x$ for every $%
x\in E$. This leads to $(q^2-q)E=0$ and hence $q^2=q$.
\end{proof}

As a consequence we get the following lemma.

\begin{lemma}
Every boolean ring is pseudo injective but not necessarily CS.
\end{lemma}

\begin{proof}
Since the quotient ring $Q=Q(R)$ is Boolean, the only invertible element of $%
Q$ is the element $1$. Hence we obtain that the units of $Q$ all belong to $%
R $. This implies that the ring $R$ is auto invariant and hence pseudo
injective. A boolean ring need not be CS as is shown in the remark below.
\end{proof}

\begin{remark}
A boolean ring need not be CS or continuous but it is always pseudo
injective (auto-invariant). We give two examples of this situation. The
first one comes from Lambek's book (\cite{Lamb}, p. 45): let $R$ be the
boolean ring of finite and cofinite susbet of the natural numbers $\mathbb{N}
$, with usual operations of intersections and unions. This ring is Boolean,
hence commutative, regular and pseudo injective. It is not CS (equivalently
continuous or quasi continuous) since its injective hull is the class of all
subsets of natural numbers.

The second example is extracted from Goodearl (cf. \cite{G}, p. 174): we
define $Q$ to be the $\mathbb{F}_2$ algebra $Q=\prod_{i\in \mathbb{R}}F_i$,
where for every real number $i$, $F_i=\mathbb{F}_2$. Let $R$ be the $\mathbb{%
F}_2$ subalgebra of $Q$ generated by $1$ and the set 
\begin{equation*}
J=\left\lbrace x\in Q \mid x_i=0 \;\mathrm{for\; all\; but\; countably \;
many\;} i\in \mathbb{R}\right\rbrace .
\end{equation*}
Then $R$ and its maximal ring of quotients $Q$ are both Boolean (hence unit
regular) and so $R$ is pseudo injective. However $R$ is not CS (equivalently
not quasi continuous) since all idempotents of $Q$ do not belong to $R$. Let
us remark that this is yet another example of a ring which is pseudo
injective but not quasi injective.
\end{remark}

Our next result is based on the following lemma. Its proof can be found in (%
\cite{JS}, lemma 3.2).

\begin{lemma}
\label{M pseudo injective and image of intersection} Let $M$ be a right
pseudo injective nonsingular module with injective hull $E(M)$. If $N$ is a
closed submodule of $E(M)$ and $\sigma$ is an injective morphism from $N$
into $E(M)$, then $\sigma(N\cap M)=\sigma(N) \cap M$.
\end{lemma}

\begin{proposition}
A regular right self pseudo injective ring is separative.
\end{proposition}

\begin{proof}
Let $Q$ be the right maxiaml quotient ring of $R$. We will again use Lemma %
\ref{going to idempotent}. So assume that $\{e_i,f_j\in R\mid 1\le i\le n,
1\le j \le l\}$ are idempotents in $R$ such that $\prod_{i=1}^n e_iQ\cong
\prod_{j=1}^l f_jQ$, and denote this isomrphism by $\sigma$. We can consider
both $\prod_{i=1}^n e_iQ$ and $\prod_{j=1}^l f_jQ$ as submodules of $Q^{n+l}$%
. We denote by $\sigma_i$ the restriction of $\sigma$ to $e_iQ$. $%
\sigma(\prod e_iR)=\prod \sigma_i(e_iR)=\prod \sigma_i(e_iQ\cap R)$. Since $%
R $ is preudo injective we use the above lemma \ref{M pseudo injective and
image of intersection} and get $\prod \sigma_i(e_iQ\cap R) =\prod
(\sigma_i(e_iQ) \cap R)=(\prod \sigma_i(e_iQ))\cap R^{n+l}=\sigma
(\prod(e_iQ))\cap R^{n+l}= \prod f_jQ \cap R^{n+l}=\prod f_jR$. Then Lemma %
\ref{going to idempotent} finishes the proof. %
\end{proof}

\vspace{3mm}

Next we will give another sufficient condition for a regular ring to be
separative. For another application of the above lemma \ref{going to
idempotent} we introduce the following definition.

\begin{definition}
\label{Closure extension property} A nonsingular ring $R$ is said to have
the closure extension property if any $R$-isomorphism between two right
ideals of $R$ 
can be extended to an isomorphism between their essential closures in $R$.
\end{definition}

\begin{proposition}
\label{Closure extension property implies separativity} Let $R$ be a regular
ring and suppose that $R$ satisfies the closure extension property. Then $R$
is separative.
\end{proposition}

\begin{proof}
Let $e_{1},\dots ,e_{n},f_{1},\dots ,f_{l}$ be idempotents in $R$, such that 
$\prod_{i=1}^{n}e_{i}Q\overset{\sigma }{\cong }\prod_{j=1}^{l}f_{j}Q$, where 
$Q=Q_{max}^{r}(R)$. According to Lemma \ref{going to idempotent}, we need to
show that $\prod_{i=1}^{n}e_{i}R\cong \prod_{j=1}^{l} f_{j}R$.
Now, $\prod_{i=1}^{n}e_{i}R$ is essential in $\prod_{i=1}^{n}e_{i}Q$ hence, $%
\sigma (\prod_{i=1}^{n}e_{i}R)$ is essential in $\sigma
(\prod_{i=1}^{n}e_{i}Q)=\prod_{j=1}^{l}f_{j}Q$. Next $\prod_{j=1}^{l}f_{j}R$
is essential in $\prod_{j=1}^{l}f_{j}Q$ and so $I^{\prime}:=\sigma
(\prod_{i=1}^{n}e_{i}R)\cap (\prod_{j=1}^{l}f_{j}R)$ is essential in $%
\prod_{j=1}^{l}f_{j}Q$ and indeed in $\prod_{j=1}^{l}f_{j}R$.  Similar arguments with $\sigma^{-1}$ gives that $I:=\sigma ^{-1}(\sigma (\prod_{i=1}^{n}e_{i}R)\cap
(\prod_{j=1}^lf_{j}R))\subseteq \prod_{i=1}^{n}e_{i}R$ is essential in $%
\prod_{i=1}^{n}e_{i}Q$ and hence essential in $\prod_{i=1}^{n}e_{i}R$. 
\begin{equation*}
\xymatrix{ {\prod_{i=1}^ne_iQ} \ar[r]^\sigma & {\prod_{j=0}^lf_jQ} \\
{\prod_{i=1}^ne_iR} \ar[u] & {\prod_{i=1}^lf_jR} \ar[u] \\ I \ar[u]
\ar[r]^\cong & I' \ar[u]}
\end{equation*}%
Since $\prod_{i=1}^{n}e_{i}R$ is closed in $R^{n}$, $\prod_{i=0}^{n}e_{i}R$
is the essential closure of $I$ and our hypothesis gives that $\sigma $
induces an isomorphism between $\prod_{i=1}^{n}e_{i}R$ and $%
\prod_{j=1}^{l}f_{j}R$, as desired.
\end{proof}

Before presenting an application, let us notice that for a regular ring the
notions of CS, quasi-continuous and continuous module are all equivalent .
In (\cite{O}) O'Meara shows that a regular continuous ring is separative.
We give here a direct proof for the CS (equivalently continuous) case.

\begin{proposition}
\label{Regular CS are separative} A regular CS (equivalently continuous or 
quasi continuous) ring $R$ is separative.
\end{proposition}

\begin{proof}
Corollary 2.32 in \cite{MM} exactly says that isomorphic submodules of a
quasi-continuous module have isomorphic closures. In our situation this
means that a regular CS ring has the closed extension property and hence is separative. 
\end{proof}

\vspace{5mm}

We do not have an example of a regular ring that is not separative.
According to comments by O'Meara (cf. \cite{O}), a good "playground" for
finding such an example would be the ring $Q=\prod_{n\in \mathbb{N}}M_{n}(F)$%
, where $F$ is a countable field. Let us remark that this ring is the
maximal quotient ring of its socle $S=\prod_{n\in \mathbb{N}}^{\ast }M_{n}(F)
$, where $\prod_{n\in \mathbb{N}}^{\ast }M_{n}(F)$ stands for sequences of $%
\prod M_{n}(F)$ with only finitely many nonzero entries.  Clearly the right (left) maximal quotient ring of a subring of $Q$ containing the socle is $Q$ itself.

\vspace{5mm}

\section{Perspective and separative rings}

\vspace{5mm}

A ring $R$ is perspective if for any two idempotents $e,f\in R$ such that $%
eR\cong fR$ there exists an idempotent $g\in R$ such that $eR\oplus gR=R=fR
\oplus gR$. This notion is left right symmetric and it is well known that a
regular ring $R$ is perspective if and only if $R$ is unit regular and hence
separative. In the previous section we tried to push down the separativity
condition from the right maximal quotient ring $Q(R)$ to $R$ itself.  In this section we will analyze the separativity condition.  

\begin{lemma}
\label{consequences of eQ oplus gQ=Q}  Let $R$ be a  von Neumann regular 
ring with right maximal quotient ring $Q$ and suppose that $e,f$ are
isomorphic idempotents of $R$. Let $g$ be an idempotent in $Q$ such that  $%
eQ\oplus gQ=Q=fQ \oplus gQ$, then there exists $e_1,f_1,g_1,g_2 \in Q$ such
that 

\begin{enumerate}
\item $ee_1+gg_1=1=ff_1+gg_2$. 

\item $gg_1e=ee_1g=ff_1g=gg_2f=0$ 

\item $e=ee_1e$, $f=ff_1f$, $ee_1Q=eQ$, $ff_1Q=fQ$, $gg_1Q=gQ$. 

\item $(1-e)gQ= (1-e)Q$ and $((1-e)gQ )\cap R = (1-e)R$. 

\item $(1-e)(gQ\cap R)\subseteq_e (1-e)R$ and $(1-f)(gQ\cap
R)\subseteq_e(1-f)R$ 

\item $E((1-e)(gQ\cap R))=(1-e)Q$ and $E((1-f)gQ\cap R))=(1-f)Q$ 

\item $eR \oplus (gQ \cap R)$ and $fR\oplus (gQ\cap R)$ are essential right
ideals in $R$. 
\end{enumerate}
\end{lemma}

\begin{proof}
(1) This is clear from $eQ\oplus gQ=Q=fQ \oplus gQ$.

(2) Right multiplying the equality $ee_1+gg_1=1$ by $e$ we get $%
e(1-e_1e)=gg_1e\in eQ\cap gQ=0$. Hence $gg_1e=0$ and $e=ee_1e$. The other
equalities are obtained similarly.

(3) These are easy facts that we leave to the reader.

(4) Left multiplying $eQ\oplus gQ=Q$ by $(1-e)$ we get $(1-e)gQ=(1-e)Q$
intersecting this equality with $R$ gives $(1-e)gQ \cap R = (1-e)R$.

(5) Let $(1-e)gx \in (1-e)gQ=(1-e)Q$, and let $E$ be an essential right
ideal of $R$ such that $\{0\}\ne gxE\subseteq gQ\cap R$. Then, since $gQ
\cap eQ=0$ we get $0\ne (1-e)gxE \subseteq (1-e)(gQ\cap R)$.  This shows
that $(1-e)(gQ\cap R)\subseteq_e (1-e)Q$ and hence  also $(1-e)(gQ \cap R)
\subseteq_e (1-e)R$, as desired.

(6) This is clear: since $(1-e)(gQ\cap R)\subseteq_e (1-e)R \subseteq_e
(1-e)Q$, we get that $E((1-e)gQ\cap R)= (1-e) Q$.

(7) Computing the injective hull we get  $E(eR\oplus gQ\cap R)=eQ\oplus gQ=Q$.  Hence $eR\oplus gQ\cap R \subseteq_e Q$ and so  $eR\oplus gQ\cap R\subseteq_e R$.
\end{proof}

 For a regular ring, perspectivity characterizes unit
regularity (cf. Corollary 4.4 \cite{G}).
The statement $(7)$ of the above lemma \ref{consequences of eQ oplus gQ=Q} leads to an equivalent characterization of unit regular rings as shown in Proposition 3.3.
But first we need the following lemma:

\begin{lemma}
Let $R$ be a regular ring and $e,f$ two idempotents in $R$ such that there exists $g\in Q=Q_{max}^r(R)$ with $eQ\oplus gQ=fQ\oplus gQ=Q$ then $eR \oplus (gQ \cap R)=fR\oplus (gQ\cap R)$ if and only if $e\in ee_1fR
$ and $f\in ff_1eR$. 
\end{lemma}
\begin{proof}
	First suppose that $eR\oplus gQ\cap R=fR\oplus gQ\cap 
	R$.  Lemma \ref{consequences of eQ oplus gQ=Q} (3) shows $gg_1Q=gQ$ and by the hypothesis $e\in eR\oplus gg_1Q$.  Hence there exist an element $g_1\in Q$ such that 
	$e=fx +gg_1y$ for some $x\in R$ and $y\in Q$ with $gg_1y \in R$.
	Since $gg_1e=0$,	we then have $0=gg_1e=gg_1fx + gg_1y$. This  gives that $gg_1y=-gg_1fx \in R$ 
	and hence,  $e=fx+gg_1y=fx-gg_1fx=(1-gg_1)fx=ee_1fx$. So that $e\in ee_1fR
	$.  The fact that $f\in ff_1eR$ is obtained similarly. Conversely,  suppose
	that $e\in ee_1fR$, then we have that $eR=ee_1fR$. We write  $%
	e=ee_1fx=(1-gg_1)fx=fx-gg_1fx\in fR\oplus gQ\cap R$.
\end{proof} 

We say that two idempotents $e,f$ in a ring $R$ are isomorphic if $eR\cong fR$.  It is well-known that this notion is left-right symmetric.  
We will give a characterization of perspectivity by using the classical module theoretic notion of complement of a submodule
We recall that for any submodule $N_R$ of $M_R$ there exists, by Zorn lemma, a maximal submodule $K_R\subseteq M$ with the property of being disjoint from $N$.  In this case $N \oplus K \subseteq_e M$.

\begin{proposition}
\label{unit regular}
	Let $R$ be a regular ring.  The following conditions are equivalent:
	\begin{enumerate}
		\item The ring $R$ is unit regular.
		\item The ring $R$ is perspective.
		\item For isomorphic idempotents $e,f\in R$,  there exists a right ideal $I$ of $R$ such that $eR\oplus I=fR\oplus I \subseteq_e R$.
		\item For isomorphic idempotents $e,f\in R$, there exists $g^2=g\in Q=Q_{max}^r(R)$ such that $eR\oplus gQ\cap R=fR\oplus gQ\cap R \subset_e R$.
	\end{enumerate}
\end{proposition}
\begin{proof}
	(1)$\Leftrightarrow$(2): This is classical (cf. \cite{G}, corollary 4.4)
	
	(2)$\Rightarrow$(3): This is clear. 
		
	(3)$\Rightarrow$(4): Starting with $eR\oplus I=fR\oplus I \subseteq_e R$ and taking injective hulls we get $eQ\oplus E(I)= fQ \oplus E(I)=Q$.  Since there exists an idempotent $g\in Q$ such that $E(I)=gQ$, we get $eQ\oplus gQ=fQ\oplus gQ=Q$.  Since $e\in fR\oplus I\subseteq fR\oplus gQ\cap R$ and $f\in eR\oplus I \subseteq eR\oplus gQ\cap R$, we get $eR\oplus gQ\cap R=fR\oplus gQ\cap R$.  
	Lemma \ref{consequences of eQ oplus gQ=Q} (7) proves (3) implies (4). 	
	
	(4)$\Rightarrow$(2): By hypothesis we have idempotents $e,f\in R$, and $g\in Q=Q_{max}^r(R)$ such that $eR\oplus gQ\cap R=fR\oplus gQ\cap R \subset_e R$.   Since $R$ is a regular ring there exist idempotents $h,p\in R$ such that $eR=hR\oplus (eR \cap fR)$ and $fR=pR \oplus (eR \cap fR)$.   Therefore we obtain 
	$eR\oplus gQ\cap R=hR\oplus (eR \cap fR) \oplus (gQ\cap R)$.  Similarly $fR\oplus (gQ\cap R)=pR \oplus (eR \cap fR)\oplus (gQ\cap R)$.  so, the hypothesis gives $hR\oplus (eR \cap fR) \oplus (gQ\cap R)=pR \oplus (eR \cap fR)\oplus (gQ\cap R)$.  This shows that $hR\cong pR$,   let $\varphi$ be the isomorphism from $hR$ to $pR$.  By setting  $S=\{x + \varphi(x) \mid x \in hR\}$,  we obtain $eR+fR=eR \oplus S=fR\oplus S$.  Since $R$ is regular we can write $R=(eR+fR)\oplus E$ for some right ideal $E$ and we obtain $R=eR\oplus (S\oplus E)=fR\oplus(S\oplus E)$.  This proves (4) implies (2).  
 \end{proof}



\begin{remark}
\label{perspectivity implies one sided unit regular}  
{\rm One would like to know what property is inherited by $R$ if $Q=Q_{max}^r (R)$ is assumed to be perspective (equivalently unit regular).

For instance, a weaker version of the closure extension property given in Section 2 will be sufficient for the perspective property: to go down from $Q_{max}(R)$ to $R$.  Let us give more details.
	
If a regular ring $R$
satisfies the property that for two isomorphic right ideals their essential closures are subisomorphic.   Then if $Q$ is perspective (equivalently unit regular) so is $R$.

	We need to show that if $e,f\in R$ are idempotents and $eR\cong fR$ then $(1-e)R$ is isomorphic to $(1-f)R$ (cf.\cite{G}). 
	Since $eR\cong fR$ we have that $eQ\cong fQ$ and the fact that $Q$ is perspective implies that there exists $g^2=g\in Q$ such that $eQ\oplus gQ=fQ\oplus gQ=Q$.
	Lemma \ref{consequences of eQ oplus gQ=Q} then shows that that $eR\oplus gQ\cap R\subseteq _{e}R$ and $fR\oplus gQ\cap R\subseteq _{e}R$. We first
	show that $(1-e)(gQ\cap R)\subseteq _{e}(1-e)R$. Let $x\in (1-e)R$, and let $%
	E$ be an essential right ideal of $R$ such that $\{0\}\neq xE\subseteq
	eR\oplus gQ\cap R$. Then, since $x=(1-e)x$, we have $0\neq
	xE=(1-e)xE\subseteq (1-e)(gQ\cap R)$, as desired. So we conclude that the
	essential closure of $(1-e)(gQ\cap R)$ in $R$ is $(1-e)R$ and similarly the
	closure of $(1-f)(gQ\cap R)$ is $(1-f)R$. Since $gQ\cap R$ is disjoint from $%
	eR$ and from $fR$, left multiplication by $(1-e)$ and by $(1-f)$ are $1-1$
	map on $gQ\cap R$ and hence $(1-e)(gQ\cap R)\cong gQ\cap R\cong (1-f)(gQ\cap
	R)$. The weak extension property then implies that $(1-e)R$ and $(1-f)R$ are subisomorphic.
	We thus have a $1-1$ map $%
	\varphi : (1-e)R\rightarrow (1-f)R$. This map with our initial isomorphism $\psi$
	gives us an injective map  $eR\oplus (1-e)R \overset{(\psi,\varphi)}{%
		\longrightarrow} fR\oplus (1-f)R$. The map $\theta=(\psi,\varphi$), has a left inverse. Since $Q$ is Dedekind finite, the same is true for $%
	End_R(R)=R$.   Thus, $\theta$ and $\varphi$ are
	isomorphisms.	This implies that $R$ is unit regular.
}
\end{remark}

We now close the paper with
the following example due to O'Meara \cite{OM} which shows that the property mentioned in the above remark is not true in general

\begin{example}
	Let $Q$ be the ring of all countably-infinite, column-finite matrices over a field $F$. Let $K=socle(Q)$. Let $R=K + D$ where $D$  is the ring of all diagonal matrices over $F$. Note $R$ is regular. Let $e=diag(1,0,1,0,...)$ be the diagonal matrix.  Now consider the right ideals of $R$:
	$I=eK$,     $I'=(1-e)K$.
	$I$ and $I'$ are isomorphic right ideals of $R$ and their essential closures in $R$ are $eQ \cap R=eR$ and $(1-e)Q \cap R= (1-e)R$.  A map from $eR$ to $(1-e)R$ would be given by left multiplication of some element of $(1-e)Re$.  Since  $(1-e)Re\subset K$,  left multiplication by an element of this set cannot induce a $1-1$ map. 
%
\end{example}

\end{document}